\documentclass{article}
\usepackage[T2A]{fontenc}
\usepackage[utf8]{inputenc}
\usepackage{amsfonts}
\usepackage{epigraph}
\usepackage[margin=1cm]{geometry}

\usepackage[english]{babel}

\usepackage{amsfonts}
\usepackage{amssymb}
\usepackage{amsmath}
\usepackage{amssymb}
\usepackage{float}

\usepackage{amsthm}
\usepackage[pdftex]{graphicx}
\usepackage{hyperref}

\usepackage{wrapfig}
\usepackage{cancel}
\usepackage{xcolor}

\newtheorem{theorem}{Theorem}
\newtheorem{definition}{Definition}
\newtheorem{conjecture}{Conjecture}
\newtheorem{lemma}{Lemma}
\newtheorem{corollary}{Corollary}
\newtheorem{proposition}{Proposition}
\newtheorem{example}{Example}

\DeclareMathOperator{\supp}{supp}
\DeclareMathOperator{\diam}{diam}

\title{Independence numbers of Johnson-type graphs}

\author{Danila Cherkashin, Sergei Kiselev}

\author{Danila Cherkashin\footnote{Chebyshev Laboratory, St.~Petersburg State University, 14th Line V.O., 29B, Saint Petersburg 199178 Russia}~\footnote{Institute of Mathematics and Informatics,
Bulgarian Academy of Sciences
Sofia 1113, 8 Acad. G. Bonchev str.}, 
Sergei Kiselev\footnote{Laboratory of Combinatorial and Geometric Structures, Moscow Institute of Physics and Technology, Institutsky lane 9, Dolgoprudny, Moscow region, 141700, Russia.}}

\begin{document}

\maketitle

\begin{abstract}
We consider a family of distance graphs in $\mathbb{R}^n$ and find its independence numbers in some cases.

Define the graph $J_{\pm}(n,k,t)$ in the following way: the vertex set consists of all vectors from $\{-1,0,1\}^n$ with exactly $k$ nonzero coordinates; edges connect the pairs of vertices with scalar product $t$.
We find the independence number of $J_{\pm}(n,k,t)$ for an odd negative $t$ and $n > n_0 (k,t)$.

\end{abstract}

\section {Introduction}

We start with common definitions. Let $G = (V,E)$ be a graph.
A subset $I$ of vertices of $G$ is \textit{independent} if no edge connects vertices of $I$.
The \textit{independence number} of a graph $G$ is the maximal size of an independent set in $G$; we denote it by $\alpha (G)$.

Generalized Johnson graphs are the graphs $J(n,k,t)$ defined as follows: the vertex set consists of vectors from the hypercube $\{0,1\}^n$ with exactly $k$ nonzero coordinates, 
edges connect vertices with scalar product $t$ (so $J(n,k,t)$ is nonempty if  $k < n$ and $2k-n \leq t < k$).
Generalized Kneser graphs $K(n,k,t)$ have the same vertex set but the edges connect vertices with scalar product at most $t$.

Now we introduce the main hero of the paper.
Define graphs $J_{\pm}(n,k,t)$ as follows: the vertex set consists of vectors from $\{-1,0,1\}^n$ with exactly $k$ nonzero coordinates, 
edges connect vertices with scalar product $t$. The graph $J_{\pm}(n,k,t)$ is nonempty if $k < n$ and $-k \leq t < k$, and also if $k = n$ and $n - t$ is even. If $t = -k$ then the graph $J_{\pm}(n,k,t)$ is a matching. Note that the edges connect vertices of the Euclidean distance $\sqrt{2(k-t)}$, which means that $J_{\pm}(n,k,t)$ is a distance graph.

Finally $K_\pm(n,k,t)$ shares the vertex set with $J_{\pm}(n,k,t)$ but the edges connect vertices with scalar product at most $t$.

\subsection{Independence and chromatic numbers of \texorpdfstring{$J(n,k,t)$}{J(n,k,t)} and \texorpdfstring{$K(n,k,t)$}{K(n,k,t)}}

Independent sets in these families of graphs are classical combinatorial objects. 
The celebrated Erd{\H o}s--Ko--Rado theorem~\cite{erdos1961intersection} determines all maximal independent sets in $J(n,k,0) = K(n,k,0)$.
Then the Frankl--Wilson theorem~\cite{frankl1981intersection}, the Frankl--F{\" u}redi theorem~\cite{frankl1985forbidding} and the Ahlswede--Khachatryan Complete Intersection Theorem~\cite{ahlswede1997complete} answered a lot of questions about the size and the structure of maximal independent sets in the graphs $J(n,k,t)$ and $K(n,k,t)$.

On the other hand a lot of questions in combinatorial geometry are related to embeddings of these graphs into $\mathbb{R}^n$. Frankl and Wilson~\cite{frankl1981intersection} used the graphs $J(n,k,t)$ to get an exponential lower bound on the chromatic number of the Euclidean space (Nelson--Hadwiger problem); Kahn and Kalai~\cite{kahn1993counterexample} used them to disprove Borsuk's conjecture.

Let us describe the picture for some small $k$ and $t$. Erd{\H o}s, Ko and Rado~\cite{erdos1961intersection} proved that $n \geq 2k$ implies
\[
\alpha[J(n,k,0)] = \binom{n-1}{k-1}.
\]
Then Lov{\' a}sz~\cite{lovasz1978kneser} proved Kneser conjecture $\chi [J(n,k,0)] = n - 2k + 2$ for $n \geq 2k$.
The following result was introduced to get a constructive bound on the Ramsey number.
\begin{proposition}[Nagy,~\cite{nag}]
\label{nagylemma}
Let $n = 4s+t$, where $0 \leq t \leq 3$. Then
\[
\alpha[J(n,3,1)] = \begin{cases} 
n      & \mbox{ if } ~ t = 0, \\ 
n-1    & \mbox{ if } ~ t = 1, \\
n-2    & \mbox{ if } ~ t = 2 \mbox{ or } 3.
\end{cases} 
\]
\end{proposition}
Then Larman and Rogers~\cite{larman1972realization} used the bound $\chi [J(n,3,1)] \geq \frac{|V [J(n,3,1)]|}{\alpha [J(n,3,1)]}$ to show that the chromatic number of the Euclidean space is at least quadratic in the dimension (initially it was proposed by Erd{\H o}s and S{\'o}s).
It turns out that the chromatic number of $J(n,3,1)$ is very close to $\frac{|V [J(n,3,1)]|}{\alpha [J(n,3,1)]}$ (and sometimes is equal to this ratio).
\begin{theorem}[Balogh--Kostochka--Raigorodskii~\cite{balogh2013coloring}]
Consider  $l \geq 2$. If $n = 2^l$ then
\[
\chi[J(n,3,1)] \leq \frac{(n - 1)(n - 2)}{6}.
\]
If $n = 2^l - 1$ then
\[
\chi[J(n,3,1)] \leq \frac{n(n - 1)}{6}.
\]
Finally, for an arbitrary $n$
\[
\chi[J(n,3,1)] \leq \frac{(n - 1)(n - 2)}{6} + \frac{11}{2}n.
\]
\end{theorem}
Tort~\cite{tort1983probleme} proved that 
\[
\chi[K(n,3,1)] = \left [ \frac{(n-1)^2}{4} \right]
\]
for $n \geq 6$.
Zakharov~\cite{zakharov2020chromatic} show that the existence of Steiner systems (see Subsection~\ref{steiner}) implies that
\[
\chi [J(n,k,t)] \leq (1+o(1)) \frac{(k-t-1)!}{(2k-2t-1)!} n^{k-t}
\]
for fixed $k > t$. 
In general $\chi [J(n,k,t)] = \Theta (n^{t+1})$ for $k > 2t + 1$ and
$\chi [J(n,k,t)] = \Theta (n^{k-t})$ for $k \leq 2t + 1$.

\subsection{Known facts about the graphs \texorpdfstring{$J_{\pm}(n,k,t)$}{J(n,k,t)} and \texorpdfstring{$K_{\pm}(n,k,t)$}{K(n,k,t)}}

For the sake of convenience we use the following combinatorial notation.
A \textit{support} of a vertex $v$ is the set of nonzero coordinates of $v$; we denote it by $\supp v$. 
Let $\mathcal{H}_k = (V_k,E_k)$ be a $k$-graph such that
\[
V(\mathcal{H}) :=  \left \{u^{sign}\,\big|\,u\in [n], sign \in \{+,-\} \right \} \quad \quad E(\mathcal{H}) := \left \{A \subset \binom{V(\mathcal{H})}{k}  \,\bigg|\, \{u^+,u^-\} \not\subset A \mbox{ for every } u  \right \} .
\]
There is a natural bijection between $E_k$ and $V(J_{\pm}(n,k,t))$.
Introduce notion \textit{signplace} for a vertex of $\mathcal H_k$ and \textit{place} for a pair of vertices $\{u^+,u^-\}$, $u \in [n]$; note that the latter definition does not depend on $k$.

From a geometrical point of view $J_{\pm}(n,k,t)$ is a natural generalization of $J(n,k,t)$.
Raigorodskii~\cite{raigorodskii2000chromatic, raigorodskii2001borsuk} used the graphs $J_{\pm}(n,k,t)$ to significantly refine
the asymptotic lower bounds in the Borsuk's problem and the Nelson--Hadwiger problem.

Unfortunately, there is no general method to find the independence number of $J_{\pm}(n,k,t)$ even asymptotically. 
One of the reasons is that the known answers have varied and sometimes rather complicated structure.
For instance the proof of the following result analogous to Proposition~\ref{nagylemma} is relatively long and the answer is quite surprising.
\begin{theorem} [Cherkashin--Kulikov--Raigorodskii, \cite{cherkashin2018chromatic}]
For $n \geq 1$ define $c(n)$ as follows:
\[
c(n) = \begin{cases} 
0    & \mbox{ if } ~ n \equiv 0 \\ 
1    & \mbox{ if } ~ n \equiv 1 \\
2    & \mbox{ if } ~ n \equiv 2 \mbox{ or } 3
\end{cases} \pmod{4}.
\]
Then 
\[
\alpha[J_\pm(n, 3, 1)] = \max \{6n-28, 4n - 4c (n)\}.
\]
\label{theo1}
\end{theorem}

In recent papers~\cite{frankl2018erdHos,frankl2018families,frankl2020intersection} Frankl and Kupavskii generalized the Erd{\H o}s--Ko--Rado theorem on some subgraphs of $J_{\pm} (n, k, t)$. 
We need additional definitions.
\[
V_{k,l} := \{v \in \{-1, 0, 1\}^n \  | \ v \mbox{ has exactly } k \ \ '1' \mbox{ and exactly } l \ \  '-1' \}.
\]

\[
 J (n,k,l,s) := (V_{k,l}, \{(v_1,v_2) \ | \ \langle v_1,v_2\rangle = s\} ).
\]
\begin{theorem}[Frankl--Kupavskii,~\cite{frankl2018erdHos}]
For $2k \leq n \leq k^2$ the equality 
\[ 
\alpha[J (n,k,1,-2)] = k \binom{n-1}{k}
\]
holds. In the case $n > k^2$ the following equality holds
\[ 
\alpha[J (n,k,1,-2)] = k \binom{k^2-1}{k} + \sum_{i=k^2}^{n-1} \binom{i}{k}.
\]
\end{theorem}
Paper~\cite{frankl2018families} deals with a more generic problem.
\begin{theorem}[Frankl--Kupavskii,~\cite{frankl2018families}]
For $2k \leq n$ the following bounds hold 
\[ 
\binom{n}{k+l}\binom{k+l-1}{l-1} \leq \alpha[J (n,k,l,-2l)] \leq \binom{n}{k+l}\binom{k+l-1}{l-1} + \binom{n}{2l}\binom{2l}{l}\binom{n-2l-1}{k-l-1}.
\]
In the case $2k \leq n \leq  3k-l$ the following equality holds 
\[ 
\alpha[J (n,k,l,-2l)] = \frac{k}{n} |V_{k,l}|.
\]

\end{theorem}

To introduce the next result, we will need the following definition.
\begin{definition}
\label{def S}
\[
S(n, D) :=
\begin{cases}
\sum_{j=0}^d \binom{n}{j} & \mbox{ if } D = 2d,\\
\binom{n-1}{d} + \sum_{j=0}^d \binom{n}{j} & \mbox{ if } D = 2d + 1.
\end{cases}
\]
\end{definition}

In \cite{frankl2016intersection} (see~\cite{frankl2019correction} for the version with a fixed mistake) Frankl and Kupavskii 
determined the independence number of $K_{\pm}(n,k,t)$ for $n > n_0(k,t)$ and found the asymptotics of the independence number of $J_{\pm}(n,k,t)$ if $t < 0$ and $n > n_0 (k,t)$.

\begin{theorem}[Frankl--Kupavskii,~\cite{frankl2019correction}]
For any $k$ and $n \geq n(k_0)$ we have:
\begin{enumerate}
    \item $\alpha[K_\pm(n, k, t)] = \binom{n-t-1}{k-t-1}$ 
    \qquad\qquad\qquad\qquad\qquad\qquad\quad for $-1 \leq t \leq k-1$,
    \item $\alpha[K_\pm(n, k, t)] = S(k, |t|-1) \binom{n}{k}$ 
    \qquad\qquad\qquad\qquad\qquad for odd $t$ such that $-k-1 \leq t < 0$,
    \item $\alpha[K_\pm(n, k, t)] = \alpha[J(n, k - \frac{|t|}{2}, \frac{|t|}{2}, t)] + S(k, |t|-2) \binom{n}{k}$
    \quad for even $t$ such that $-k-1 \leq t < 0$.
\end{enumerate}
\label{FK_Kneser} 
\end{theorem}

\begin{theorem}[Frankl--Kupavskii,~\cite{frankl2016intersection}]
For any $t < 0$ and $n > n_0(k, t)$ we have 
\[
\alpha[J_\pm(n,k,t)] \leq S(k, |t|-1) \binom{n}{k} + O \left (n^{k-1}\right ).
\]
\label{approx} 
\end{theorem}
The main technique in the Frankl--Kupavskii theorems is shifting, which is not applicable directly to Johnson-type graphs.

\subsection{Results}
\label{res}

Let $J(n, k, even)$ be a graph with the vertex set $\{0, 1\}^n$, where edges connect vertices with even scalar product (note that each vertex has a loop if $k$ is even). Define $J(n, k, odd)$ in a similar way. 
Let $J_{\pm}(n, k, even)$ and $J_{\pm}(n, k, odd)$ be defined analogously to $J(n, k, even)$ and $J(n, k, odd)$.

\begin{proposition}
If $n > n_0(k)$ then
\[
\alpha[J_{\pm}(n, k, even)] = 2^k \alpha[J(n, k, even)],
\]
\[
\alpha[J_{\pm}(n, k, odd)] = 2^k \alpha[J(n, k, odd)].
\]
\label{prop_parity}
\end{proposition} 
For $n > n_0(k)$ the exact values of $\alpha[J_{\pm}(n, k, even)]$ and $\alpha[J_{\pm}(n, k, odd)]$ are determined in Theorem~\ref{oddevengraphs}.

\begin{proposition}
\label{prop_steiner}
For every $n \geq k$ we have
\[
\alpha[J_\pm(n, k, k-1)] = 2^{k} \alpha[J(n, k, k-1)].
\]
\end{proposition}

Note that $\alpha[J(n, k, k-1)]$ is the size of a largest partial Steiner $(n, k, k-1)$-system. In particular, if the divisibility conditions hold, then $\alpha[J(n, k, k-1)] = \binom{n}{k-1} / k$ (see Subsection~\ref{steiner}).

We use the Katona averaging method and Reed--Solomon codes to prove the following theorem. 

\begin{theorem}
Suppose that $n > k2^{k+1}$.
Then 
\[
\alpha[J_\pm(n, k, -1)] = \binom{n}{k}.
\]
\label{-1}
\end{theorem}

Theorem~\ref{-1} can be generalized as follows.
\begin{theorem}
Suppose that $t$ is a negative odd number, $n > n_0(k)$.
Then 
\[
\alpha[J_\pm(n, k, t)] = S \left (k, |t|-1 \right) \binom{n}{k},
\]
where $S$ is defined in Definition~\ref{def S}.
\label{bsimple}
\end{theorem}

The next theorem is a consequence of Theorems~\ref{EL} and~\ref{-1}.

\begin{theorem}
Let $n > \frac{9}{2} k^3 2^{k}$. Then 
\[
\alpha[J_\pm(n, k, 0)] = 2\binom{n-1}{k-1}.
\]
\label{0}
\end{theorem}

One can extract a kind of stability version of the previous theorem from its proof. 
\begin{corollary}
\label{cor0}
Suppose that $I$ is an independent set in $J_{\pm}(n,k,0)$ and no place intersects all the vertices of $I$. Then
\[
|I| \leq C(k) \binom{n}{k-2}.
\]
\end{corollary}

\paragraph{Structure of the paper.}
In Section~\ref{tools} we describe several classical definitions and theorems, that are used in examples and proofs:
Katona averaging methods, nontrivial intersecting families, isodiametric inequality for the Hamming cube,
simple hypergraphs and Reed--Solomon codes, Steiner systems and finally families with intersections of prescribed parity.

Section~\ref{examples} contains examples, Section~\ref{proofs} provides proofs, Section~\ref{small} specifies the results in the case $k \leq 3$. We finish with open questions in Section~\ref{open}.

\section{Tools}
\label{tools}

\subsection{Katona averaging method}
\label{KatonaLemma}
Properties of a graph with a rich group of automorphisms sometimes can be established via consideration of a proper subgraph.
The following lemma is a special case of Lemma 1 from~\cite{katona1975extremal}.

\begin{lemma} [Katona,~\cite{katona1975extremal}] Let $G = (V,E)$ be a vertex-transitive graph. Let $H$ be a subgraph of 
$G$. Then
\[
\frac{\alpha(G)}{|V(G)|} \leq \frac{\alpha(H)}{|V(H)|}.
\]
\label{Katona}
\end{lemma}

For example Lemma~\ref{Katona} immediately implies that for every fixed $k,t$ the following decreasing sequences converge
\[
a_n := \frac{\alpha[J_{\pm}(n,k,t)]}{|V[J_{\pm}(n,k,t)]|} \quad \mbox{and} \quad b_n := \frac{\alpha[K_{\pm}(n,k,t)]}{|V[K_{\pm}(n,k,t)]|},
\]
as $J_\pm(n-1, k, t)$ and $K_\pm(n-1, k, t)$ are isomorphic to subgraphs of $J_\pm(n, k, t)$ and $K_\pm(n, k, t)$, respectively, and both $J_\pm(n, k, t)$ and $K_\pm(n, k, t)$ graphs are clearly vertex-transitive.

Also since $J(n,k,t)$ is a subgraph of $J_{\pm}(n,k,t)$, Lemma~\ref{Katona} implies
\[
\frac{\alpha[J_{\pm}(n,k,t)]}{|V[J_{\pm}(n,k,t)]|} \leq \frac{\alpha[J(n,k,t)]}{|V[J(n,k,t)]|},
\]
which gives by $|V [J_\pm(n,k,t) ]|  = 2^k \binom{n}{k} = 2^k |V[J(n,k,t)]|$ the following bound
\begin{equation}
\alpha[J_{\pm}(n,k,t)] \leq 2^k \alpha[J(n,k,t)].
\label{averagingbound}
\end{equation}
It turns out that bound~\eqref{averagingbound} is rarely close to the optimal. On the other hand sometimes it is tight, for instance in Propositions~\ref{prop_parity} and~\ref{prop_steiner}.

\subsection{Nontrivial intersecting families}
\label{ELTh}

A family of sets $\mathcal{A}$ is \textit{intersecting} if every $a, b \in A$ have nonempty intersection.
A \textit{transversal} is a set that intersects each member of $\mathcal{A}$.

\begin{theorem}[Erd{\H o}s--Lov{\'a}sz,~\cite{erdos1975problems}]
Let $\mathcal{A}$ be an intersecting family consisting of $k$-element sets.
Then at least one of the following statements is true:
\begin{itemize}
    \item [(i)] $\mathcal{A}$ has a transversal of size at most $k-1$;
    \item [(ii)] $|\mathcal{A}| \leq k^k$.
\end{itemize}
\label{EL}
\end{theorem}
One can find better bounds in the case~(ii)~\cite{cherkashin2011maximal,arman2017upper,frankl2017near,zakharov2020size}.
In particular, for $k = 3$ it is known that $3^3 = 27$ in~(ii) can be replaced with 10 and this result is sharp~\cite{frankl1996covers}.

\begin{theorem}[Deza,~\cite{deza1974solution}]
Let $\mathcal{A}$ be a family of $k$-element sets such that $|A_i \cap A_j|$ is the same for all different $A_i,A_j \in \mathcal{A}$. Then at least one of the following statements is true:
\begin{itemize}
    \item [(i)] $A_i \cap A_j$ is the same for all different $A_i,A_j \in \mathcal{A}$;
    \item [(ii)] $|\mathcal{A}| \leq k^2 - k + 1$.
\end{itemize}
\label{DezaThm}
\end{theorem}

\subsection{An isodiametric inequality}
\label{kleitman}
Define \textit{the Hamming distance} between two subsets of $[n]$ as the size of their symmetric difference.
The \textit{Hamming distance} between two vectors $v_1,v_2\in\{-1,0,1\}^n$ is the number of coordinates that differ between $v_1$ and $v_2$.
The \textit{diameter} of a family $\mathcal{A} \subset 2^{[n]}$ or $\mathcal{A} \subset \{-1,0,1\}^n$ is the maximal distance between its members.

\begin{theorem}[Kleitman,~\cite{kleitman1966combinatorial}]
Let $\mathcal{A} \subset 2^{[n]}$ be a family with the diameter at most $D$ for $n > D$. Then
\[
|\mathcal{A}| \le
S(n, D),
\]
where $S$ is defined in Definition~\ref{def S}.
\label{Kl}
\end{theorem}

Theorem~\ref{Kl} is sharp: in the case of even $D$ the equality holds for the family $\mathcal K(n, D) := \{A\subset[n]\colon |A| \leq \frac{D}{2}\}$ and in the case of odd $D$ the equality holds for the family $\mathcal K_x(n, D) := \{A\subset[n]\colon |A \setminus \{x\}| \leq \frac{D}{2}\}$ for some fixed $x\in[n]$. 
Moreover, in \cite{frankl2017stability} Frankl proved the following stability result.
\begin{theorem}[Frankl,~\cite{frankl2017stability}]
Let $\mathcal{A} \subset 2^{[n]}$ be a family with the diameter at most $D$ and $|\mathcal A| = S(n, D)$ for $n \ge D + 2$. Then in the case of even $D$ family $\mathcal A$ is a translate of $\mathcal K(n, D)$ (that is, equal to $\{A \Delta T\colon A\in \mathcal K(n, D)\}$ for some $T \subset [n]$) and in the case of odd $D$ family $\mathcal A$ is a translate of $\mathcal K_y(n, D)$.
\label{Kl_stability}
\end{theorem}

\subsection{Simple hypergraphs and Reed--Solomon codes}
\label{codes}
A hypergraph $H=(V,E)$ is a collection of \textit{(hyper)edges} $E$ on a finite set of \textit{vertices} $V$.
A hypergraph is called \textit{$k$-uniform} if every edge has size $k$.
A hypergraph is \textit{simple} if every two edges share at most one vertex. 
The following construction is a special case of Reed--Solomon codes; it is also known as Kuzjurin's construction~\cite{kuzjurin1995difference}. 

Fix a prime $p > k$ and let the vertex set $V$ be the union of $k$ disjoint copies of 
a field with $p$ elements $\mathbb{F} = GF(p)$; call them $\mathbb{F}_1, \dots, \mathbb{F}_k$.
Consider the following system of linear equations
\[
\sum^k_{i=1} i^j x_i = 0, \quad j = 0, 1, \dots , k - 3
\]
over $\mathbb{F}_p$.
The solutions $\{x_1,\dots x_k\} \in \mathbb{F}_1 \sqcup \dots \sqcup \mathbb{F}_k$, where $x_i \in \mathbb{F}_i$, form edge set $E$.
Fixing two arbitrary variables there is a unique solution over $\mathbb{F}_p$, because the corresponding
square matrix is a Vandermonde matrix with nonzero determinant. It means that there are $p^2$ different solutions and $|e_1 \cap e_2| \leq 1$ for every distinct $e_1,e_2 \in E$.
Summing up, $H_p(k) := (V,E)$ is a $p$-regular $k$-uniform simple hypergraph with $|V| = pk$ and $|E| = p^2$.

A $k$-uniform hypergraph is \textit{$b$-simple} if every two edges share at most $b$ vertices. 
The same construction with $k-b-1$ equations gives an example of a $k$-uniform $b$-simple hypergraph $H(p,k,b)$.

Further we use \textit{regularity} of $H(p,k,b)$ in the following sense. Consider an arbitrary $b$-vertex set $A$. If $A$ contains at most 1 vertex from every copy of $\mathbb{F}_p$ then $H$ has exactly $p$ hyperedges containing $A$; otherwise $H$ contains no such edges. Slightly abusing the notation we say that $b$-codegree of $H$ is $p$.

\subsection{Steiner systems}
\label{steiner}

A \textit{Steiner system} with parameters $n$, $k$ and $l$ is a collection of $k$-subsets of $[n]$ such that every $l$-subset of $[n]$ is contained
in exactly one set of the collection.
There are some obvious necessary ‘divisibility conditions’ for the existence of Steiner $(n,k,l)$-system: 
\[
\binom{k-i}{l-i} \mbox{ divides } \binom{n-i}{k-i} \mbox{ for every } 0 \leq i \leq k-1.
\]

In a breakthrough paper~\cite{keevash2014existence} Keevash proved the existence of Steiner $(n, k, l)$-systems for fixed $k$ and $l$ under the divisibility conditions and for $n > n_0(k,l)$ (different proofs can be found in~\cite{glock2016existence, keevash2018existence}).

\paragraph{Partial Steiner system.}  When the divisibility conditions do not hold we are still able to construct a large \emph{partial Steiner system}, that is, a collection of $k$-subsets of $[n]$ such that every $l$-subset of $[n]$ is contained in \emph{at most} one set of the collection.
R{\"o}dl confirmed a conjecture of Erd\H{o}s and Hanani and proved the following theorem.
\begin{theorem}[R{\"o}dl, \cite{rodl1985packing}]
For every fixed $k$ and $l < k$, and for every $n$ there exists a partial $(n,k,l)$-system with 
\[
(1 - o(1))\binom{n}{l}/\binom{k}{l}
\] 
$k$-subsets.
\end{theorem}
Later the result was refined in \cite{kostochka1998partial,grable1999more,kim2001nearly}. Also it follows from the mentioned results on Steiner systems.

\subsection{Families with even or odd intersections}

Recall that $J(n, k, even)$ and $J(n, k, odd)$ were defined in Subsection~\ref{res}. Frankl and Tokushige determined the independence numbers of these graphs.
 
\begin{theorem}[Frankl--Tokushige,~\cite{frankl2016uniform}]
\label{oddevengraphs}
Let $n \geq n_0(k)$. Then
\begin{align*}
    \alpha[J(n, k, odd)] =&~ \binom{\lfloor n / 2 \rfloor}{k / 2}
    &\mbox{ for even }k,
    \\
    \alpha[J(n, k, even)] =&~ \binom{\lfloor (n - 1) / 2 \rfloor}{(k - 1) / 2}
    &\mbox{ for odd }k.
\end{align*}

\end{theorem}

In the case when $k$ is even, the equality is achieved for the following family: we split $[n]$ into pairs and take all sets consisting of $k/2$ pairs. In the case when $k$ is odd we also add a fixed point $x\in[n]$ to each constructed set.

\section{Examples}
\label{examples}

Let us start with a simple example which is rarely close to the independence number.
\begin{example}
Let $t < 0$, $k > |t|$. Then $\alpha[J_\pm(n, k, t)] \geq 2^{|t| - 1}\binom{n}{k}$.
\label{ex1}
\end{example}

\begin{proof}
Fix an ordering of the coordinates. Take all vertices of $J_\pm(n,k,t)$ with the first $k-|t|+1$ nonzero coordinates equal to 1. Then two vertices can have different signs on at most $|t| - 1$ positions, therefore their scalar product is at least $-|t| + 1 = t + 1$.
\end{proof}

The following example is a part of Theorem~\ref{FK_Kneser}.
\begin{example}
\label{exmain}
For any  $t < 0$ and $k > |t|$ we have
\[
\alpha[J_\pm(n, k, t)] \geq S(k, |t|-1) \binom{n}{k},
\]
and for even $t$ we also have
\[
\alpha[J_\pm(n, k, t)] \geq S(k, |t|-1) \binom{n}{k} + \binom{k-1}{|t|/2}.
\]

\label{someex}
\end{example}

\begin{proof}
We start with the first bound for the case of odd $t$.
Let $I_{odd}$ be the set of all vertices of $J_\pm(n,k,t)$ with at most $(|t| - 1) / 2$ negative entries. Each $k$-set is the support of exactly
\[
\sum_{j=0}^{(|t| - 1) / 2} \binom{k}{j} = S(k,|t|-1)
\]
vertices in $I_{odd}$. Any two vectors in $I_{odd}$ may differ in at most
$2 (|t| - 1) / 2 = |t| - 1$ coordinates, so their scalar product is at least $t + 1$, and $I_{odd}$ is an independent set of the desired size.

Now we deal with the case of even $t$.
Fix an ordering of the coordinates. For every $k$-set $f$ add to $I_{even}$ all the vertices with at most $|t|/2 - 1$ negative entries on $f$ and all the vertices with $-1$ on the last coordinate of $f$ and exactly $|t|/2 - 1$ other coordinates. 
Then each $k$-set is the support of exactly
\[
\sum_{j=0}^{|t|/2 - 1} \binom{k}{j} + \binom{k-1}{|t|/2-1} = S(k,|t|-1)
\]
vertices in $I_{even}$.
Assume that $I_{even}$ is not independent, i.e. the scalar product of some $v_1, v_2 \in I_{even}$ is equal to $t$. 
Then $v_1$ and $v_2$ together have at least $|t|$ negative entries.
Hence both $v_1$ and $v_2$ have exactly $|t|/2$ negative entries, so both $\supp v_1$ and $\supp v_2$ have $-1$ at the last coordinates $x_1$ and $x_2$ respectively.
But then both $v_1$ and $v_2$ can not have $+1$ at coordinates $x_2$ and $x_1$ respectively, so the scalar product is at least $t+1$.
This contradiction shows that $I_{even}$ is an independent set of the desired size.

Now we proceed to the second bound.
Let us add to $I_{even}$ all the vertices on the lexicographically first support $\{1, \ldots, k\}$ with exactly $|t|/2$ negative entries and having $+1$ at the $k$-th coordinate.
Obviously the resulting set $I$ has the claimed size.
Since $k > |t|$ no edge connects two vertices from $I$ on the support $\{1, \ldots, k\}$.

Consider a vertex $v$ from $I_{even}$ and an added vertex $u$. 
Note that $u$ and $v$ together have at most $|t|$ negative entries. 
Since the largest coordinate of $\supp v$ has the number greater than $k$ and $v$ has $-1$ in this coordinate, the scalar product of $u$ and $v$ is at least $t+1$. Thus $I$ is independent.

\end{proof}

\begin{example}
For $t \geq 0$ we have
\[
\alpha[J_{\pm}(n, k, t)] \geq 2 \alpha[J(n, k, t)].
\]
\label{example_positive_t_1}
\end{example}

\begin{proof}
Let $I \subset V[J(n, k, t)]$ be an independent set of size $\alpha[J(n, k, t)]$. 
Define $I_\pm$ as a subset of $V[J_\pm(n, k, t)]$ consisting of vertices with all positive or all negative entries on every 
support $f = \supp v$, $v \in I$.
It is easy to see that the subset $I_\pm$ is independent in $J_\pm(n, k, t)$.
\end{proof}

\section{Proofs}
\label{proofs}

\paragraph{Proof of Proposition~\ref{prop_parity}.} Let $parity$ stand for $odd$ or $even$.

To prove the lower bounds consider an arbitrary maximal independent set $I$ in the graph $J(n,k,parity)$. 
Then all the vertices on the supports from $I$ form an independent set $I_\pm$ in $J_\pm(n,k,parity)$.
So
\[
\alpha[J_\pm(n,k,parity)] = 2^k \alpha[J(n,k,parity)].
\]

The upper bounds simply follows from Lemma~\ref{Katona}, since $J(n, k, parity)$ is a subgraph of $J_\pm(n, k, parity)$.

\paragraph{Proof of Proposition~\ref{prop_steiner}.} Since $J(n, k, k-1)$ is a subset of $J_\pm (n,k,k-1)$, by Lemma~\ref{Katona} we have
\[
\alpha[J_\pm(n, k, k-1)] \leq  2^{k} \alpha[J(n, k, k-1)].
\]

To prove the lower bound consider an arbitrary maximal independent set $I$ in the graph $J(n,k,k-1)$. 
Then all the vertices on the supports from $I$ form an independent set $I_\pm$ in $J_\pm(n,k,k-1)$.

\subsection{Proof of Theorem~\ref{-1}}

We start with the lower bound. One can take the vertices only with non-negative coordinates (so exactly one vertex on each support is taken); obviously the scalar product of such vertices is always non-negative, so
\[
\alpha(n,k,-1) \geq \binom{n}{k}.
\]

Now we will show the upper bound. Denote $G := J_{\pm}(n,k,t)$.
Fix a prime $p$, $n / (2k) \le p \le n / k$ (so by the statement of the theorem $p > 2^k$), 
and let $H := H_p(k)$ (see Subsection~\ref{codes}) be a $p$-regular $k$-uniform simple hypergraph with $V(H) \subset [n]$.
Define graph $G[H]$ as a subgraph of $G$, consisting of the vertices with support on edges of $H$.
So we have
\[
|V(G[H])| = 2^k|E(H)|.
\]

Fix an independent set $I$ in $G[H]$; consider the set $X \subset [n]$ of coordinates on which the vertices from $I$ have both signs.
Denote by $\supp I$ the set of all supports of vertices from $I$ ($\supp I \subset E(H)$) and
for a given $e \in \supp I$ put $e_X := e \cap X$. 

Note that $I$ has at most $2^{|e_X|}$ vertices on the support $e$ ($|e_X|$ might be zero).
Hence 
\begin{equation}
|I| \leq \sum_{e \in \supp I} 2^{|e_X|} = \sum_{e \in E(H): |e_X| = 0} 2^{|e_X|} + \sum_{e \in E(H): |e_X| > 0} 2^{|e_X|} =
\left|\left\{e \in E(H): |e_X| = 0\right\}\right| + \sum_{e \in E(H): |e_X| > 0} 2^{|e_X|}. 
\label{eq11}
\end{equation}

Let us show that $e_X$ form a disjoint cover of $X$. 
Suppose the contrary, i.e. there are $e,f \in \supp I$ such that $e_X \cap f_X \neq \emptyset$. Since hypergraph $H$ is simple, and $e, f$ correspond to its hyperedges, we have $|e_X \cap f_X| = |e \cap f| = 1$. Put $\{u\} := e \cap f$. By the definition of $X$ there are vertices $v_1$, $v_2 \in I$ having different signs on $u$. Since $I$ is independent, $v_1$ and $v_2$ have the same support (say, not $f$).
So every vertex of $G[H]$ with support $f$ forms an edge in $G[H]$ with one of $v_1$ or $v_2$, thus $I$ is not independent; contradiction.

So $\sum |e_X| = |X|$. Since the sequence $2^k/k$, $k \geq 1$, is non-decreasing and 
\[
\frac{a_1 + a_2 + \ldots + a_t}{b_1 + b_2 + \ldots + b_t} \leq \max\left( \frac{a_1}{b_1}, \frac{a_2}{b_2}, \dots, \frac{a_t}{b_t} \right),
\]
we have
\begin{equation}
\sum_{e \in E(H): |e_X| > 0} 2^{|e_X|} \leq \frac{|X|}{k} 2^k.
\label{eq12}    
\end{equation}

By definition, every $e \in \left\{e \in E(H): |e_X| = 0\right\}$ has empty intersection with $X$.
Since $H$ is $p$-regular, $X$ intersects at least $\frac{p|X|}{k}$ edges of $H$ (since every $k$-edge is counted at most $k$ times), so
\begin{equation}
    \left|\left\{e \in E(H): |e_X| = 0\right\}\right| \leq |E(H)| - \frac{p|X|}{k}.
    \label{eq13}
\end{equation}
Summing up, by~\eqref{eq11},~\eqref{eq12},~\eqref{eq13} and the choice of $p$, we have
\[
|I| \leq |E(H)| - \frac{|X|}{k}p  + \frac{|X|}{k} 2^k \leq |E(H)|,
\]
which implies $\alpha(G[H]) \leq |E(H)|$, hence
\[
\frac{V(G[H])}{\alpha (G[H])} \geq 2^k.
\]
By the definition $G[H]$ is a subgraph of the graph $G$, so Lemma~\ref{Katona} finishes the proof.

For some $k$ one can choose a smaller $H$ and require a weaker inequality for $n$, 
for instance in the case $k=3$ (see Subsection~\ref{betterH}).

\subsection{Proof of Theorem~\ref{bsimple}}
This is a generalization of the proof of Theorem~\ref{-1}. The lower bound is provided in the first part of Example~\ref{exmain}.

Denote $G = J_{\pm}(n,k,t)$ during the proof.
The case $t = -k$ is obvious, because $J_\pm(n,k,-k)$ is a matching. From now $|t| \leq k-1$.
Fix $n$ and a prime $p \le n/k$ to be large enough. 
Let $H = H(p, k, |t|)$ (see Subsection~\ref{codes}) be a $k$-uniform $|t|$-simple hypergraph with $|t|$-codegree $p$. 
Fix an embedding of $V(H)$ into $[n]$.

Define $G[H]$ as a subgraph of $G$, consisting of all the vertices with support on edges of $H$.
Fix an independent set $I$ in $G[H]$. 

Let an \textit{object} $O$ be a pair of opposite vectors $\{o,-o\}$ of length $\sqrt{|t|}$ with $\{0, \pm 1\}$ entries.
Let $\mathcal{X}$ be the set of objects $O = \{o,-o\}$
such that $(v_1,o) = (v_2,-o) = |t|$ for some vertices $v_1,v_2 \in I$ (this means that $v_1$ and $v_2$ coincide on $\supp o$ with $o$ and $-o$, respectively). Since $H$ is $|t|$-simple, one has $\supp v_1 = \supp v_2$.

Let $E_{tight}$ be the set consisting of such edges $e \in E(H)$ that $\diam I[e] < |t|$, where $I[e]$ stands for the set of vertices of $I$ with the support $e$.
Put $E_{wide} := E(H) \setminus E_{tight}$. Let $I_{tight}$ and $I_{wide}$ stand for the sets of vertices of $I$ with the support from $E_{tight}$ and $E_{wide}$ respectively. Then $|I| = |I_{tight}| + |I_{wide}|$.

Consider an arbitrary support $e \in E_{wide}$; by the definition of $E_{wide}$ there are an object $X=\{x,-x\} \in \mathcal{X}$ and vertices $v_1,v_2 \in I[e]$, such that $(v_1,x) = (v_2,-x) = |t|$. Since $I$ is independent and $H$ is $|t|$-simple, distinct $e_1$ and $e_2 \in E_{wide}$ cannot lead to the same $X \in \mathcal{X}$, so
\begin{equation}
|\mathcal{X}| \geq |E_{wide}|.
\label{xewide}    
\end{equation}

For every $e \in E_{tight}$ we have $\diam I[e] < |t|$, thus Theorem~\ref{Kl} implies
\begin{equation}
    |I[e]| \leq S \left (k, |t|-1 \right).
\label{123}
\end{equation}

\medskip

Fix a support $f\subset[n]$, $|f| = k$, and consider a family $A \subset \{-1,1\}^f$ with the diameter at most $|t|-1$ of the size $S(k,|t|-1)$. 
Also consider an object $O = \{o,-o\}$ such that $\supp O \subset f$ (recall that $|\supp O| = |t|$). By the pigeon-hole principle
one of $\{o,-o\}$ has at most $(|t|-1)/2$ negative entries. Thus there is a vector $v$ from $\mathcal{K}(k,|t|-1)$ such that
$(v,o) = |t|$ or $(v,-o) = |t|$. By Theorem~\ref{Kl_stability} $A$ is a translate of $\mathcal{K}(k,|t|-1)$, 
so the previous conclusion holds for $A$.

Fix an object $X\in \mathcal{X}$ and consider an arbitrary support $e \in E_{tight}$ containing $\supp X$. 
Assume that $(v,x) = \pm t$ for some $v \in I[e]$. Consider a support $g \in E_{wide}$ such that there are $u_1, u_2 \in I[g]$, satisfying $(u_1,x) = (u_2,-x) = t$ ($g$ exists because $X\in \mathcal{X}$). Since $H$ is $|t|$-simple, $(u_1,v) = t$ or $(u_2,v) = t$; a contradiction.
By Theorem~\ref{Kl_stability} we can refine the bound~\eqref{123} in this case:
\begin{equation}
\label{refined}
    |I[e]| \leq S \left (k, |t|-1 \right)-1.
\end{equation}

By the construction of $H$ for every $X \in \mathcal{X}$, $\supp X$ is contained in exactly $p$ edges of $H$ (because it is contained in at least one edge).
Every edge of $H$ is the support of $\binom{k}{|t|}2^{|t|-1}$ objects, so is counted above at most $\binom{k}{|t|}2^{|t|-1}$ times. By~\eqref{xewide} at most $|\mathcal{X}|$ of the edges are wide. So the refined bound~\eqref{refined} is applicable to at least
\[
\frac{p|\mathcal{X}|}{\binom{k}{|t|}2^{|t|-1}} - |\mathcal{X}|
\]
tight edges. Then
\[
|I_{tight}| \leq S \left (k, |t|-1 \right) |E_{tight}| - \frac{p|\mathcal{X}|}{\binom{k}{|t|}2^{|t|-1}} + |\mathcal{X}|.
\]
On the other hand there is a straightforward bound
\[
|I_{wide}| \leq 2^k |E_{wide}| \leq 2^k |\mathcal{X}|.
\]
Putting it all together
\begin{equation}
|I| = |I_{tight}| + |I_{wide}| \leq S \left (k, |t|-1 \right) |E_{tight}| - \frac{p|\mathcal{X}|}{\binom{k}{|t|}2^{|t|-1}} + (2^k+1)  |\mathcal{X}|.
\label{alphahahaha}
\end{equation}

For a large $n$ (then $p$ is also large enough) the inequality~\eqref{alphahahaha} implies $\alpha(G[H]) \leq S \left (k, |t|-1 \right) |E(H)|$. 
By the definition $G[H]$ is a subgraph $G$ and
\[
\frac{\alpha (G[H])}{V(G[H])} \leq \frac{S \left (k, |t|-1 \right)}{2^k},
\]
so Lemma~\ref{Katona} finishes the proof.

\subsection{Proof of Theorem~\ref{0}}



Consider an arbitrary independent set $I$ in the graph $J_{\pm}(n,k,0)$. 
Note that supports of the vertices of $I$ form an intersecting family; denote it by $F$. Let $U$ be a minimal (by inclusion) transversal of $F$.
As $U$ is minimal, for every coordinate $a \in U$ there is a vertex $x_a \in I$, such that $\supp x_a \cap U = \{a\}$.

In the case $|U| > 1$ we can consider the set 
\[
C := U \cup \supp x_a \cup \supp x_b
\]
for two different $a, b \in U$. Since $\supp x_a \cap \supp x_b \neq \emptyset$ we have $|C| \leq 3k$
and every $f \in F$ intersects $C$ in at least two places (suppose that $|f \cap U| = 1$, then it should intersect either $(\supp x_a)\setminus U$ or $(\supp x_b) \setminus U$). 
Hence
\[
|I| \leq 2^k \binom{|C|}{2} \binom{n}{k-2} < 2^k \frac{9k^2}{2} \binom{n}{k-2}.
\]
Recall that $n > \frac{9}{2} k^3 2^k$, so 
\[
2^k\frac{9k^2}{2} \binom{n}{k-2} < 2^k\frac{9k^2}{2} \frac{n^{k-2}}{(k-2)!} < \frac{n}{k-1} \frac{n^{k-2}}{(k-2)!} < 2 \binom{n-1}{k-1}.
\]

The remaining case is $|U| = 1$, say $U = \{u\}$.
Consider only vertices containing $u^+$, by Theorem~\ref{-1} we have at most $\binom{n-1}{k-1}$ such vertices.
The same bound for $u^-$ gives the desired bound.
The union of all vertices with a constant sign intersecting a fixed place $u$ achieves this bound.

\subsection{Proof of Corollary 1}

Let us repeat the proof of Theorem~\ref{0}. Let $I$ be an arbitrary independent set in $J_{\pm}(n,k,0)$.
Then 
\[
|I| < 2^k \frac{9k^2}{2} \binom{n}{k-2}
\]
or the family of all supports of vertices from $I$ has a transversal of size 1. The first two possibilities implies
\[
|I| \leq C(k) \binom{n}{k-2};
\]
the last one contradicts the condition of the corollary.

\section{The case \texorpdfstring{$k \leq 3$}{k <= 3}}
\label{small}

We have implemented {\"O}sterg{\aa}rd algorithm~\cite{ostergaard2002fast} to find independence numbers of several small graphs.
All the calculations were done on a standard laptop in a few hours. The source can be found in~\cite{KiselevIndependence}.

\subsection{The case \texorpdfstring{$k = 2$}{k = 2}}
\label{k=2}

\paragraph{The case $t=-1$.} By simple calculations we have
\[
\alpha[J_\pm(2,2,-1)] = \alpha[J_\pm(3,2,-1)] = 4, \quad
\alpha[J_\pm(4,2,-1)] = 8, \quad \alpha[J_\pm(5,2,-1)] = 10.
\]
In Section~\ref{KatonaLemma} we show that the sequence
\[
\frac{\alpha[J_{\pm}(n,2,-1)]}{|V[J_{\pm}(n,2,-1)]|}
\]
is non-increasing, so
\[
\alpha[J_\pm(n, 2, -1)] = \binom{n}{2}
\]
for $n \geq 5$.

\paragraph{The case $t=0$.}
It is straightforward to check that
\[
\alpha[J_\pm(2,2,0)] = 2, \quad
\alpha[J_\pm(3,2,0)] = \alpha[J_\pm(4,2,0)] = 6.
\]
For the case $n > 4$ we can repeat the proof of the Theorem~\ref{0} and show that $\alpha[J_\pm(n, 2, 0)] = 2(n-1)$.

\paragraph{The case $t=1$.} From Proposition~\ref{prop_steiner} we have 
\begin{align*}
\alpha[J_\pm(n,2,1)] &= 2n & \mbox{ for even $n$},
\\
\alpha[J_\pm(n,2,1)] &= 2(n-1) & \mbox{ for odd $n$}.
\end{align*}

\subsection{The case \texorpdfstring{$k=3$, $t=-1$}{k = 3, t = -1}}
\label{betterH}

\begin{proposition}
Let $n \geq 7$. Then
\[
\alpha[J_\pm(n,3,-1)] = \binom{n}{3}.
\]
\label{ngeq7}
\end{proposition}

\begin{proof}
Fano plane is the projective plane over $GF(2)$ i.e. the following simple 3-graph on 7 vertices
\[
\{1,2,3\}, \{1,4,7\}, \{1,5,6\}, \{2,4,6\}, \{2,5,7\}, \{3,4,5\}, \{3,6,7\}.
\]
Consider an arbitrary embedding $F$ of the Fano plane into $V[J_\pm(n,3,-1)]$. As usual consider the subgraph $G[F]$; it has $7\cdot 2^3 = 56$ vertices.
One may check by hands or via computer that $\alpha(G[F]) = 7$. By Lemma~\ref{Katona}  
\[
\alpha[J_{\pm}(n,3,-1)] \leq \binom{n}{3}.
\]
On the other hand, Example~\ref{ex1} implies $\alpha[J_{\pm}(n,3,-1)] = \binom{n}{3}$.
\end{proof}

By the computer calculations we have
\[
\alpha[J_\pm(6,3,-1)] = 21 > \binom{6}{3} = 20,
\]
so Proposition~\ref{ngeq7} is sharp. Also
\[
\alpha[J_\pm(5,3,-1)] = 14, \quad \alpha[J_\pm(4,3,-1)] = 8, \quad \alpha[J_\pm(3,3,-1)] = 2.
\]

\subsection{The case \texorpdfstring{$k=3$, $t=0$}{k = 3, t = 0}}

By the computer calculations we have
\[
\alpha[J_\pm(3,3,0)] = \alpha[J_\pm(4,3,0)] = 8, \quad \alpha[J_\pm(5,3,0)] = 20, \quad \alpha[J_\pm(6,3,0)] = 32,
\]
\[
\alpha[J_\pm(7,3,0)] = \alpha[J_\pm(8,3,0)] = \alpha[J_\pm(9,3,0)] = 56.
\]

\begin{proposition}
Let $n \geq 9$. Then
\[
\alpha[J_\pm(n,3,0)] = 2\binom{n-1}{2}.
\]
\label{ngeq9}
\end{proposition}

\begin{proof}
The example is inherited from Theorem~\ref{0}.

Let us proceed with the upper bound. For the case $n = 9$ the computer calculations give us the desired result.
Let us repeat the proof of Theorem~\ref{0}, updating it for small values of $n$.
Let $I$ be a maximal independent set in $G := J_{\pm}(n,3,0)$.

Clearly supports of vertices of $I$ form a 3-uniform intersecting family.
Theorem~\ref{EL} states that an intersecting family either contains at most 27 sets or has a 2-transversal.
It is known~\cite{frankl1996covers} that the constant 27 can be refined to 10. 

In the first case the family of supports has no 2-transversal. Then $|I| \leq 8\cdot 10$, which is enough for $n > 10$.
Assume the contrary to the statement in the case $n = 10$, id est $|I| > 72$. It implies that vertices in $I$ have exactly $10$ different supports.
Suppose that every pair of supports shares exactly one vertex.
Then by Theorem~\ref{DezaThm} all the supports have one common vertex, so at least $1 + 2\cdot 10 > 10$ coordinates are required.
Thus there are supports $f_1$, $f_2$ such that $|f_1\cap f_2| = 2$.
The initial graph $G$ has 16 vertices with supports $f_1$ and $f_2$; by the equality $\alpha[J_\pm(4,3,0)] = 8$, 
$I$ has at least 8 missing vertices on these supports. 
This refines the bound $|I| \leq 80$ to the desired $|I| \leq 72 = 2\binom{9}{2}$.

In the second case we have a one-point transversal set, say $U = \{u\}$.
Let $I_{sign}$ be a set of vertices from $I$ containing $u^{sign}$, where $sign \in \{+,-\}$. 
Clearly $|I| = |I_+| + |I_-|$.
After removing coordinate $u$ from every vertex, $I_+$ becomes an independent set in $J_\pm(n-1,2,-1)$.
By the Subsection~\ref{k=2} $|I_+| \leq \binom{n-1}{2}$. The same bound for $I_-$ finishes the proof in this case.

In the last case we have a transversal set of size 2, say $\{a,b\}$.
Let $I_a$ be the set of vertices of $I$ containing $a$ and not containing $b$, $I_b$ is defined analogously. 
Both $I_a$ and $I_b$ are nonempty, otherwise there is a one-point transversal set which is the previous case.
Define $I_{ab} = I \setminus I_a \setminus I_b$.
Computer calculations show that for $n = 10$ we have at most 48 vertices in an independent set with such conditions.

Let $n$ be greater than 10; for every set $A \subset [n]$, such that $|A|=10$ and $a,b \in A$, we have $\alpha (G[A]) \leq 48$ (here $G[A]$ stands for the subgraph of $G$ containing all the vertices $v$ such that $\supp v \subset A$). Define $I[A]$ as the set of vertices $i$ from $I$ such that $\supp i \subset A$; note that $I[A]$ is an independent set.
Every vertex from $I_{ab}$ belongs to $\binom{n-3}{7}$ different $A$, every vertex from $I_a \cup I_b$ belongs to $\binom{n-4}{6}$ different $A$.
Summing up inequalities $|I[A]| \leq \alpha(G[A]) \leq 48$ over all choices of $A$ we got
\[
\binom{n-3}{7} |I_{ab}| + \binom{n-4}{6} (|I_a| + |I_b|) \leq 48 \binom{n-2}{8}
\]
which is equivalent to
\[
\frac{n-3}{7} |I_{ab}| + (|I_a| + |I_b|) \leq \frac{48}{56} (n-2)(n-3).
\]
Finally,
\[
|I| = |I_{ab}| + |I_a| + |I_b| \leq \frac{n-3}{7}|I_{ab}| + (|I_a| + |I_b|) \leq \frac{48}{56} (n-2)(n-3) < 2 \binom{n-1}{2}.
\]
\end{proof}

\subsection{The case \texorpdfstring{$k=3$, $t=-2$}{k = 3, t = -2}}

Example~\ref{someex} gives us a lower bound $\alpha[J_\pm(n, 3, -2)] \geq 2\binom{n}{2} + 2$.
Note that the Katona averaging method does not give exact result because of the additional term of a smaller order of growth.

First, note that Theorem~\ref{approx} in this case gives the bound
\[
\alpha[J_\pm(n, 3, -2)] \leq 2\binom{n}{3} + 8\binom{n}{2}.
\]
Indeed, let $I$ be an independent set in $J_{\pm}(n, 3, -2)$. We call a vertex $v\in I$ \emph{bad} if there is another vertex with the same support which differs in exactly two places. Otherwise we call a vertex \emph{good}. From Theorem~\ref{Kl} there are at most $2\binom{n}{3}$ good vertices.

Let us show that the number of bad vertices is at most $8\binom{n}{2}$.
Indeed, each bad vertex has a pair of signplaces such that antipodal pair of signplaces contained in another vertex. But then all vertices containing one of these two pairs of signplaces must have the same third place therefore there are at most $8 \binom{n}{2}$ bad vertices.
\qed

Using more accurate double counting we can prove the following upper bound.

\begin{proposition}
For $n \geq 6$ we have
\[
\alpha[J_\pm(n, 3, -2)] \leq 2\binom{n}{3} + \frac{8}{3} \binom{n}{2}.
\]
\end{proposition}

\begin{proof}
A pair of vertices $v, w \in I$ is called \emph{tangled} if these vertices have the same support and differ exactly at two places.
Define the \emph{weight} $c_I(v, i, j)$, where $v\in I$ and $i,j\in v$, in the following way:
\[
c_I(v, i, j) = \begin{cases}
1, \mbox{ if $v$ does not have tangled vertices in $G$}, \\
2, \mbox{ if $v$ has a tangled vertex in $G$ which differs at places $i,j$}, \\
0.5, \mbox{ otherwise}.
\end{cases}
\]
Note that for a vertex $v$ sum of corresponding weights is at least 3.
Let $d_{i,j}$ be the sum of weight of vertices containing places $i$ and $j$ and let us estimate an upper bound for $d_{i, j}$.
Then there are three cases which depend on whether there are tangled vertices containing places $i, j$ and whether these vertices have antipodal signs on places $i, j$.

In the first case there are no tangled vertices in $I$ which differ in places $i, j$. Then for any place $l$ the total weight of vertices with support $\{i, j, l\}$ is at most 2. Then $d_{i,j} \leq 2(n-2)$.
In the second case there are tangled vertices in $I$ which contain all four pairs of signplaces on places $i, j$. Then there are at most 8 vertices containing these places and $d_{i,j} \leq 16$.

In the last case there are two vertices in $I$ which are antipodal on places $i, j$ and there are no vertices in $I$ which contain one of the pairs of signplaces on places $i, j$. Then there are at most 4 vertices which differ in places $i, j$ and their total weight is at most $8$. The rest of vertices containing places $i, j$ have the same signs on these places therefore their total weight is at most $2(n-2)$.

Therefore, $d_{i,j} \leq 2n + 4$ and 
\[
3|I| \leq \sum_{1\leq i < j \leq n} d_{i,j} \leq \binom{n}{2} (2n+4) = 6\binom{n}{3} + 8\binom{n}{2}.
\]
\end{proof}

\section{Open questions}
\label{open}

It seems very challenging to find a general method providing the independence number of $J_{\pm} (n,k,t)$.
Here we discuss questions that seem for us both interesting and relatively easy.

\paragraph{Small values of the parameters.} The smallest interesting case is $J_{\pm}(n,3,-2)$. We hope that for $n > n_0$ Example~\ref{someex} is the best possible, i.e. 
\[
\alpha[J_{\pm}(n,3,-2)] = \alpha[K_{\pm}(n,3,-2)] = 2\binom{n}{3} + 2.
\]
Recall that the last equality is established by Theorem~\ref{FK_Kneser}.

Another small case leads to the following conjecture.
\begin{conjecture}
Let $n > n_0$ be an even number. Then
\[
\alpha [J_{\pm}(n,4,1)] = 2n(n-2).
\]
\end{conjecture}
Obviously $\alpha [J_{\pm}(n,4,1)] \geq \alpha [J_{\pm}(n,4,odd)] = 2n(n-2)$ (see Proposition~\ref{prop_steiner}).

\paragraph{Chromatic numbers.}  
Usually finding or evaluating the chromatic number is a more complicated problem than finding or evaluating the independence number.
In particular Lov{\' a}sz~\cite{lovasz1978kneser} proved Kneser conjecture on the chromatic number of $K(n,k,0)$ 17 year after 
Erd{\H o}s, Ko and Rado determined the independence number of this graph.

Looks like the simplest case is $J_{\pm}(n,k,0)$, since Theorem~\ref{0} gives precise value of the independence number for large $n$ and Corollary~\ref{cor0} provides a stability-type result.
It turns out that for $k \leq 3$ the chromatic number is linear in $n$ (the precise value is also of interest).
Unfortunately, for $k \geq 4$ we can only say that
\[
c(k)n \leq \frac{|V[J_{\pm} (n,k,0)]|}{\alpha[J_{\pm} (n,k,0)]} \leq  \chi [J_{\pm} (n,k,0)] \leq  \frac{|V[J_{\pm} (n,k,0)]|}{\alpha[J_{\pm} (n,k,0)]} \log |V[J_{\pm} (n,k,0)]| \leq C(K) n \log n
\]
for some positive constants $c(k)$, $C(k)$.
The last inequality holds since $J_{\pm} (n,k,0)$ is a vertex-transitive graph (see~\cite{lovasz1975ratio}).

\paragraph{Difference between $J_{\pm}(n,k,t)$ and $K_{\pm}(n,k,t)$.}
It turns out that for a negative odd $t$ Theorems~\ref{FK_Kneser} and~\ref{bsimple} give
\[
\alpha[J_{\pm}(n,k,t)] = \alpha[K_{\pm}(n,k,t)].
\]
Does it hold for all negative $t$? Do we have 
\[
\chi[J_{\pm}(n,k,t)] = \chi[K_{\pm}(n,k,t)]
\]
in this case? 

The general comparison of the behavior of independence numbers and chromatic numbers of these graphs is also of interest.

\paragraph{Acknowledgments.} We thank A.~Kupavskii for the patient explanation of his results.
We thank G.~Kabatiansky for pointing out the notion of Reed--Solomon codes.
We are grateful to J.~Balogh, G.O.H.~Katona, F.~Petrov and A.~Raigorodskii for their attention to our work.
This research is supported by the Russian Science Foundation under grant 16-11-10039.

\bibliographystyle{plain}
\bibliography{main}

\end{document}